\documentclass[reqno]{amsart}

\usepackage{amscd}
\usepackage{amsmath}
\usepackage{amsfonts}
\usepackage{amssymb}
\usepackage{graphicx}
\usepackage{enumerate}
\usepackage{tikz}
\usepackage{pinlabel}
\usepackage{verbatim}
\usepackage{bm}
\usepackage{color}
\usepackage{epstopdf}

\newcommand{\diffeo}{\cong}

\newcommand{\tri}{\mathcal{T}}
\newcommand{\trib}{\mathcal{B}}

\newcommand{\lef}{\mathcal{F}}

\theoremstyle{plain}
\newtheorem{theorem}{Theorem}

\newtheorem{lemma}{Lemma}

\theoremstyle{definition}
\newtheorem{definition}{Definition}
\newtheorem{example}{Example}

\theoremstyle{remark}
\newtheorem{question}{Question}
\newtheorem{remark}{Remark}

\numberwithin{equation}{section}
\everymath{\displaystyle}

\title{Trisecting Smooth $4$--dimensional Cobordisms}
\author{Nickolas A. Castro}

\begin{document}

\begin{abstract}
We extend the theory of relative trisections of smooth, compact, oriented $4$--manifolds with connected boundary given by Gay and Kirby in \cite{gay} to include $4$--manifolds with an arbitrary number of boundary components. Additionally, we provide sufficient conditions under which relatively trisected $4$--manifolds can be glued to one another along diffeomorphic boundary components so as to induce a trisected manifold. These two results allow us to define a category \textbf{Tri} whose objects are smooth, closed, oriented $3$--manifolds equipped with open book decompositions, and morphisms are relatively trisected cobordisms. Additionally, we extend the Hopf stabilization of open book decompositions to a relative stabilization of relative trisections.
\end{abstract}
\maketitle
\begin{section}{Introduction}\label{S:intro}
A trisection of a smooth, compact, connected, oriented $4$--manifold $X$ is a decomposition $X = X_1 \cup X_2 \cup X_3$ into three diffeomorphic $4$--dimensional $1$--handlebodies ($X_i \cong \natural^k S^1 \times B^3$) with certain nice intersection properties. Trisections are the natural $4$--dimensional analog of Heegaard splittings of $3$--manifolds, and there are striking similarities between the two theories. Gay and Kirby first introduced trisections for compact manifolds with connected boundary in ~\cite{gay} and showed that all such manifolds admit a trisection. In the closed case, $\partial X_i \cong \#^k S^1 \times S^2$ is given a genus $g$ Heegaard splitting $(X_i \cap X_{i+1}) \cup (X_i \cap X_{i-1}).$ Such a $(g,k)$--trisection can be given to every closed $4$--manifolds for some $g$ and $k.$ The case when $\partial X \neq \emptyset$ is much more involved; a portion of each $\partial X_i$ must be glued to the other pieces and what remains contributes to $\partial X.$ The key feature of trisections relative to a non-empty boundary is the fact that they induce open book decompositions on the bounding $3$--manifold(s). This fact, first proved in the case of connected boundary by Gay and Kirby, has given a great deal of insight to the theory.

The main results of this paper are: (1) the existence and uniqueness (rel. boundary) of trisections of smooth $4$--manifolds with $m >1$ boundary components; (2) the gluing theorem for relative trisections; and (3) the existence of trisection stabilizations relative to a non-empty boundary component. The first result extends the definition given in \cite{gay} to ensure that relative trisections induce open book decompositions on each boundary component of $X.$ It is this induced boundary structure which gives rise to the gluing theorem, allowing us to glue relative trisections along boundary components with compatible induced open books. Relative stabilizations allow us to obtain new relative trisections from old ones in a way that modifies the induced open book decomposition of a chosen boundary component via a Hopf stabilization. This was inspired by a modification of Lefschetz fibrations which also stabilizes the open book on the boundary (see \cite{gompf}, \cite{osb}). We now state the main results.
\begin{theorem}[Existence and Uniqueness]\label{T:existence}
Every smooth, compact, connected, oriented $4$--manifold $X$  admits a trisection $\tri_X$ relative to its boundary. Additionally, $\tri_X$ induces an open book decomposition on each of the components of $\partial X.$ Moreover, given any collection of open book decompositions on $\partial X,$ there exists a relative trisection of $X$ which induces the given open books. This relative trisection is unique up to (interior) stabilization.
\end{theorem} 

\begin{theorem}[Gluing Theorem]\label{T:Gluing}
Let $X$ and $W$ be smooth, compact, connected oriented $4$--manifolds with non-empty boundary equipped with relative trisections $\tri_X$ and $\tri_W$ respectively. Let $B_X \subset \partial X$ be any collection of boundary components of $X$ and $f:B_X \hookrightarrow \partial W$ an injective, smooth map which respects the induced open book decompositions $\tri_X|_{B_X}$ and $\tri_W|_{f(B_Z)}.$ Then $f$ induces a trisection $\tri = \tri_X \underset{f}{\cup} \tri_W$ on $X \underset{f}{\cup} W.$
\end{theorem}
An outline of the paper is a follows: Section~\ref{S:Background} provides a brief discussion of the preliminaries; trisections of closed $4$--manifolds, open book decompositions and Lefschetz fibrations. Sections~\ref{S:basicdefs} and \ref{S:gluingtheorem} prove Theorems~\ref{T:existence} and \ref{T:Gluing} respectively. In section~\ref{S:relstab} we discuss stabilizing trisections relative to a chosen boundary component; this is the only section which requires knowledge of Lefschetz fibrations. We conclude with section~\ref{S:Conclusion} wherein we make a few remarks on the theory.
\end{section}

\begin{section}{Trisections of Closed $4$--manifolds, Open Book Decompositions, and Lefschetz Fibrations\label{S:Background}}
Here we will briefly discuss the preliminaries. Trisections of closed $4$--manifolds are much easier to define than trisections relative to a non-empty boundary. However, the intuition lends itself quite nicely to the relative case.
\begin{definition}\cite{gay}
A \emph{$(g,k)$--trisection } of a closed $4$--manifold $X$ is a decomposition $X = X_1 \cup X_2 \cup X_3$ such that for each $i$
	\begin{enumerate}[i)]
		\item $X_i$ is diffeomorphic to $\natural^k S^1 \times B^3,$
		\item $(X_i \cap X_{i+1}) \cup (X_i \cap X_{i-1})$ is a genus $g$ Heegaard splitting of $\partial X_i,$
	\end{enumerate}
where indices are taken mod $3$.
\end{definition}
As a consequence, the triple intersection $X_1 \cap X_2 \cap X_3 = F_g$ is a genus $g$ surface, called the \emph{trisection surface.} Additionally, a handle decomposition of $X$ tells us that $\chi(X)= 2+ g -3k.$ This tells us two things. The first is that for any given manifold $X$, $k$ is determined by $g$ which allows us to refer to a $(g,k)$-trisection as a \emph{genus $g$ trisection.} The second fact is that the genera of any two trisections of a fixed $X$ must be equivalent mod 3. We will occasionally denote a trisection of $X$ by $\tri_X$, or $\tri.$

\begin{figure}[ht!]
\centering{
	\labellist
		\pinlabel $X$ at 190 380
		\pinlabel \resizebox{10pt}{!}{$X_1$} at 58 361
		\pinlabel \resizebox{10pt}{!}{$X_2$} at 96 369
		\pinlabel \resizebox{10pt}{!}{$X_3$} at 99 350
		\pinlabel \resizebox{10pt}{!}{$X'_1$} [r] at 40 164
		
		\pinlabel \resizebox{6pt}{!}{$\alpha$} at 118 297
		\pinlabel \resizebox{10pt}{!}{$N_\alpha$} at 74 135
	\endlabellist
\includegraphics[scale=.6]{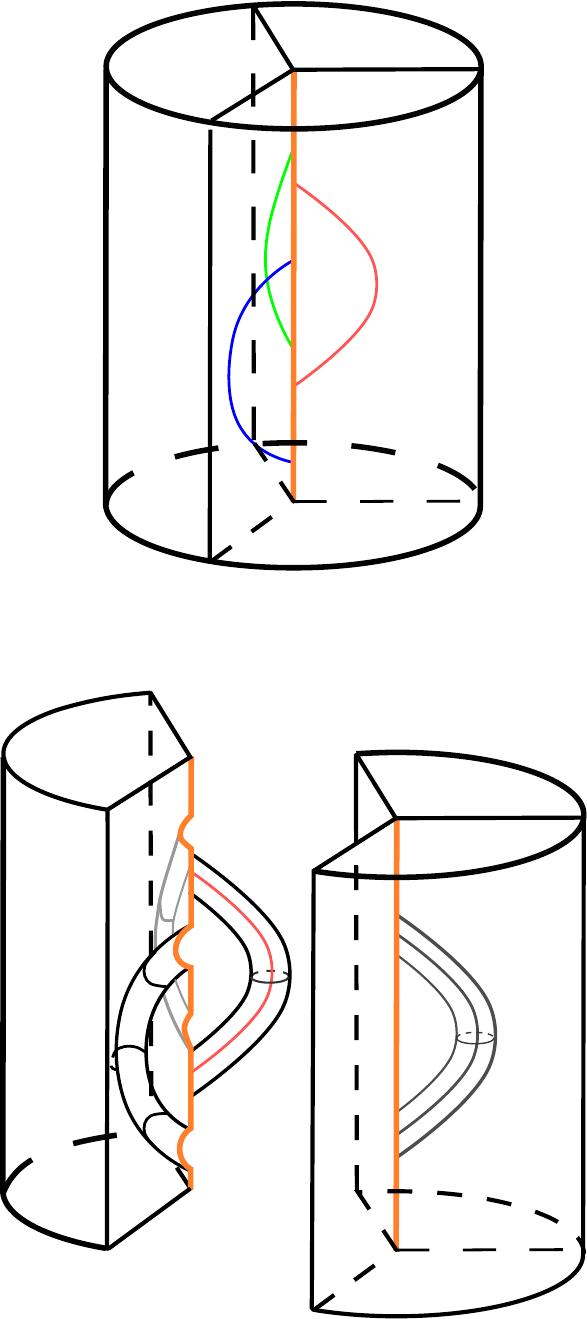}
}
	\label{F:stabilizingcylinder}
	\caption{Stabilizing a Trisection}
\end{figure}
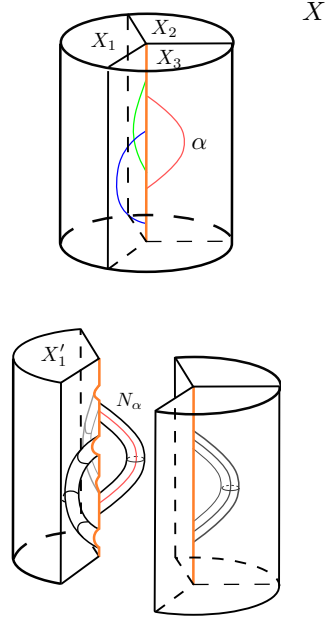

Stabilizing a trisection is a bit more complex than stabilizing a Heegaard splitting. However, we still obtain a new trisection $\tri'$ by modifying $\tri$ in the most trivial way possible. Choose a boundary parallel, properly embedded arc $\alpha \subset X_2 \cap X_3$ and a regular neighborhood $N_1 \subset X_2 \cap X_3$ of $\alpha.$ Choose arcs $\beta, \gamma$ and their neighborhoods $N_2 \subset X_1 \cap X_3$ and $N_3 \subset X_1 \cap X_2$ similarly. We define the pieces of our new trisection to be
	\begin{align*}
		X'_1:=& \overline{X_1 \cup N_1}\setminus (N_2 \cup N_3)\\
		X'_2:=& \overline{X_2 \cup N_2}\setminus (N_1 \cup N_3)\\
		X'_3:=& \overline{X_3 \cup N_3}\setminus (N_1 \cup N_2)
	\end{align*}
Attaching the $1$--handles $N_i$ to $X_i$ results in the boundary connected sum with $S^1 \times B^3.$ However, removing the other two neighborhoods from $X_i$ do not change its topology. This is due to the fact that each curve lies in the intersection of two pieces of our trisection. This has the effect of ``digging a trench'' out of $X_i.$ On the other hand, each one of these neighborhoods are attached to the trisection surface which increases the genus of the trisection by three. This should be expected from the equation $\chi(X)= 2+ g -3k.$

The following theorem is the trisection analog of the Reidermeister-Singer Theorem for Heegaard splittings. 
\begin{theorem}[Gay-Kirby, 2012~\cite{gay}] \label{T:closeduniqueness}
Every smooth, closed, connected, oriented $4$--manifold admits a trisection. Moreover, any two trisections of the same $4$--manifold become isotopic after a finite number of stabilizations.
\end{theorem}

As mentioned above, a relative trisection induces a structure on the bounding $3$--manifold known as an \emph{open book decomposition}. See \cite{etnyre} for a detailed introduction to open books. 
\begin{definition}
An \emph{open book decomposition} of a connected $3$--manifold $M$ is a pair $(B, \pi)$ such that $B$ is a link in $M$ called the \emph{binding} and $\pi: M \setminus B \rightarrow S^1$ is a fibration such that the closure of the fibers $\overline{\pi^{-1}(t)} = \Sigma_t,$ called \emph{pages}, are genus $g$ surfaces with $\partial \Sigma_t = B$ for every $t.$
\end{definition} It is a well known result of Alexander that every $3$--manifold admits an open book decomposition \cite{alexander}. An \emph{abstract open book} is a pair $(\Sigma, \phi)$, where $\Sigma$ is a surface with boundary and $\phi \in \textrm{\emph{Diff}}_+(\Sigma, \partial \Sigma).$ If we construct the mapping torus $\Sigma_\phi$, we can attach $S^1 \times D^2$ to each boundary component so that $\partial(\Sigma \times \{t\}) = \underset{b}{\sqcup} S^1 \times \{t\} \in S^1 \times \partial D^2$. The result is a closed $3$--manifold $M_\phi$ equipped with an open book decomposition with pages $\Sigma$ and binding given by the cores of the solid tori. We will use both types of open books, depending on which one better suits our needs.

Open book decompositions can also be stabilized. Given an abstract open book $(P, \phi),$ choose a properly embedded arc $\alpha \subset P.$ Attach a $2$--dimensional $1$--handle to $\partial\alpha \times I \subset \partial P,$ giving a new surface $P'.$ The co-core of the $1$--handle together with $\alpha$ comprise a simple closed curve $\gamma \subset P',$ which we require to have page framing $\mp1.$ Define the new abstract open book $(P', \tau^{\pm}_{\gamma} \circ \phi)$, where $\tau^{\pm}_\gamma$ denotes a positive/negative Dehn twist about $\gamma.$ This process is called a \emph{positive/negative Hopf stabilization} of $(P, \phi).$ It is a standard result that $M_\phi \cong M_{\tau_{\gamma} \circ \phi}$. The page $P'$ can also be viewed as the result of plumbing a Hopf band onto $P$ along $\alpha$. 

\begin{definition}
Let $S$ and $X$ be smooth, compact, connected, oriented manifolds of dimension $2$ and $4$ respectively. A \emph{Lefschetz fibration} on $X$ is a map $f:X \rightarrow S,$ such that
	\begin{enumerate}[i)]
		\item $f$ has finitely many critical points  $\Gamma=\{p_1, \ldots, p_n\}\subset int(X)$ such that $f(p_i) \neq f(p_j)$ for $i \neq j$
		\item around each critical point $f$ can be locally modeled by an orientation preserving chart as $f(u, v) = u_1^2 + v_2^2.$
		\item in the complement of the singular fibers, $f^{-1}(f(\Gamma)),$ $f$ is a smooth fibration with fibers $F$
	\end{enumerate}
The fibers of critical values are said to be \emph{singular} and all other fibers are \emph{regular}. Removing the condition that charts preserve orientation results in what is known as an \emph{achiral Lefschetz fibration.}
\end{definition}
Like relative trisections, Lefschetz fibrations over the disk with bounded fibers also induce open book decompositions on the boundary. Additionally, it is a straight forward process to obtain a relative trisection of $X$ from such a Lefschetz fibration. Thus, we will restrict our attention to Lefschetz fibrations over $D^2$ whose regular fibers are surfaces with boundary.

The critical points $p_i$ correspond to $2$--handles attached to $F \times D^2$ along simple, closed curves $\gamma_i$ called \emph{vanishing cycles}. The page framings of the $2$--handles are $-1$ or $+1$, depending on whether the local models of the singularities are orientation preserving or reversing respectively. The induced open book decomposition is then given by $(F, \tau^{\pm}_{\gamma_1} \circ \cdots \tau^{\pm}_{\gamma_n})$.

We can obtain a new Lefschetz fibration $\lef'$ from $\lef$ by attaching a $4$--dimensional canceling $1-2$ pair to $X$ as follows: Attach the $1$--handle $h^1$ so that the attaching sphere lies in the binding of the open book decomposition of $\partial X$ induced by $f;$ thus attaching a $2$--dimensional $1$--handle to each of the fibers. The cancelling $2$--handle $h^2$ is then attached to one of these fibers along an embedded curve with page framing $\pm1$ which intersects the co-core of the new $1$--handle exactly once. This ensures that $h^2$ corresponds to a Lefschetz singularity. This modification defines a new Lefschetz fibration $\lef'$ on $X$ whose pages and singular values differ from that of $\lef$ as above. Moreover, the global monodromy of $\lef'$ is given by $\tau^{\mp1} \circ \phi$, where $\phi$ is the monodromy of $\lef.$ This modification induces a Hopf stabilization of the open book decomposition of $\partial X$ induced by $\lef.$ (For more details see \cite{gompf}, \cite{osb}.)
\end{section}

\begin{section}{Trisecting Cobordisms\label{S:basicdefs}}
Just as in the closed case, a trisection of a $4$--manifold $X$ with non-empty boundary is a decomposition $X = X_1 \cup X_2 \cup X_3$ where $X_i \cong \natural^k S^1 \times B^3$, for some $k,$ such that the $X_i$'s have ``nice'' intersections. Before the proper definition can be stated, we will discuss the model pieces to which the $X_i$'s and their intersections will be diffeomorphic.

We begin with $F_{g,b},$ a connected genus $g$ surface with $b$ boundary components, and attach $n$ $3$--dimensional $2$--handles to $F_{g,b} \times \{1\} \subset F_{g,b} \times [0,1]$ along a collection of $n$ essential, disjoint, simple, closed curves. If $\partial X$ has $m$ connected components, then we require that surgery on $F_{g,b}$ along the curves separate $F_{g,b}$ into $m$ components, none of which are closed. Such a $3$--manifold $C$ is called a \emph{compression body}. We define our model pieces $Z := C \times I \cong \natural^k S^1 \times B^3,$ where $k = 2g +b-1 - n.$

\begin{remark}\label{D:compressionbody} In general, a \emph{compression body} is a $3$--manifold which is the result of attaching $0$ and $1$ handles to $\Sigma \times I,$ where $\Sigma$ is a compact surface, with or without boundary. In what follows we will only be dealing with compression bodies such as $C.$
\end{remark}

It is sometimes convenient to consider a Morse function $f: C \rightarrow [0,1]$ with $f^{-1}(0) = F_{g,b}$ and $f^{-1}(1)= P,$ the ``other end'' of our compression body. We will denote $P = \underset{i = 1}{\overset{m}{\sqcup}} P_i,$ where $P_i \cong F_{p_i, b_i},$ $\underset{i}{\Sigma} p_i = g-(n-(m-1))$ and $\underset{i}{\Sigma} b_i = b.$ The function $f$ will only have $n$ index--$2$ critical points. Let us arrange for the first $n-(m-1)$ critical points to have distinct critical values such that passing through these critical levels does not increase the number of components of the level sets. Let us further arrange for the remaining $m-1$ critical points (which each correspond to a separating $2$--handle) to have the same critical value. A schematic for this construction is given in Figure~\ref{F:constructing},where the red lines represent the critical levels of $f$)

\begin{figure}[ht!]
\centering{
	\labellist
		\pinlabel {$C \; \cong$} [r] at 10 55
		\pinlabel $P_1$ at 18 162
		\pinlabel $P_m$ at 188 162
		\pinlabel $F_{g,b}$ [t] at 100 -2
		
		\pinlabel {$\cong \; Z$} [l] at 545 60
		\pinlabel $F_{g,b}$ [t] at 440 -2
		\pinlabel \rotatebox{-45}{\resizebox{!}{18pt}{$]$}} at 528 10
		\pinlabel $0$ at 522 -18
		\pinlabel $1$ at 555 20
		
		\pinlabel $\cdots$ at 103 115
		
		\pinlabel $\cdots$ at 440 116
		
		\pinlabel \rotatebox{30}{$\leftarrow$} [l] at 175 82
		\pinlabel \resizebox{!}{5pt}{$m-1$ separating} [l] at 200 94
		\pinlabel \resizebox{!}{5pt}{$2$--handles} [l] at 215 80
		\pinlabel \resizebox{5pt}{15pt}{$\}$} at 182 29
		\pinlabel \resizebox{!}{6pt}{non-separating} [l] at 185 41
		\pinlabel \resizebox{!}{6pt}{$2$--handles} [l] at 198 25
	\endlabellist
\includegraphics[scale=.5]{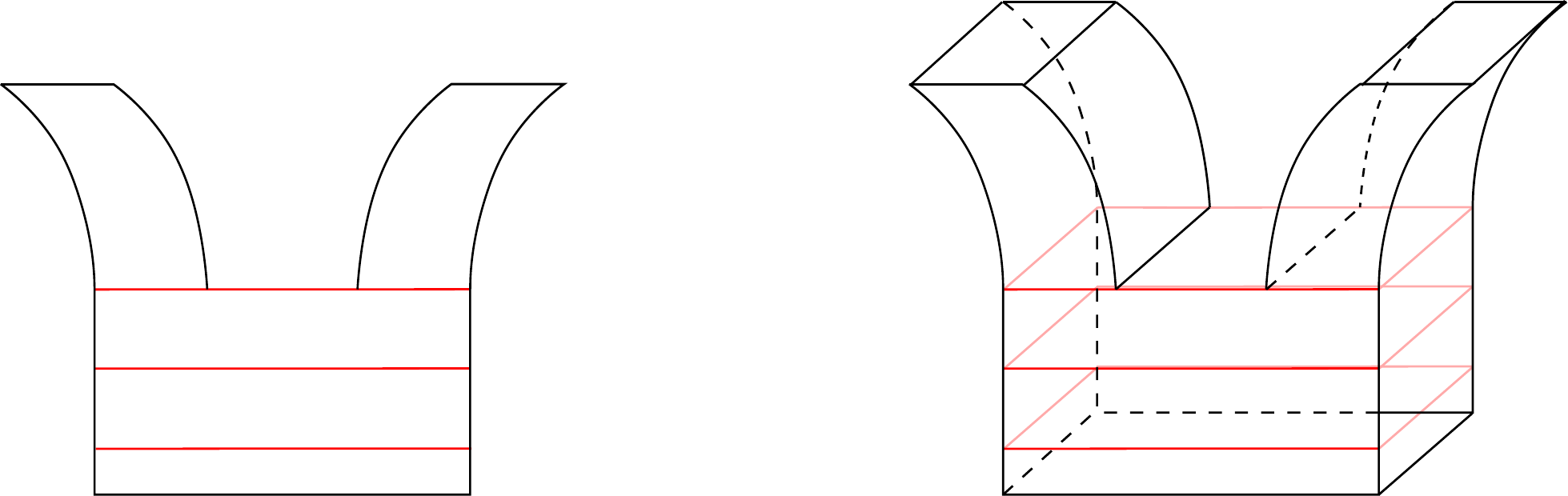}
}
\vspace{.2cm}
	\caption{Constructing the Model Pieces\label{F:constructing}}
\end{figure}
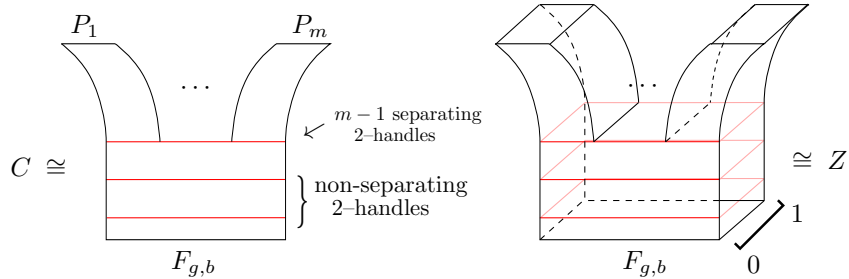

Note that by constructing $C$ upside down, it becomes immediately clear that $C$ is a $3$--dimensional handlebody: We attach $n$ $3$--dimensional $1$--handles to $P \times I,$ ensuring to connect every component. Since each $P_i \times I$ is a neighborhood of a punctured surface, we have that $P_i \times I \cong \natural^{l_i}S^1 \times D^2,$ where $l_i = 2p_i +(b_i-1).$ Thus, attaching $1$--handles in the prescribed manner gives us
	$$C \cong \natural^k S^1 \times D^2,$$
where $k =\underset{i}{\Sigma}l_i + n - (m-1)=2g + b-1 - n.$ (We will regularly make use of the fact that our compression bodies are $3$--dimensional handlebodies, for which it is essential that $F_{g,b}$ has non-empty boundary.) Thus, 
	$$Z \cong \natural^k S^1 \times B^3.$$
For the intersections consider $\partial Z,$ which we decompose into two pieces,
	$$\begin{array}{rl}
		\partial_{In} Z := & \left(C \times \{0\} \right) \cup \left(F_{g,b} \times I \right) \cup \left(C \times \{1\}\right)\cr
		\partial_{Out} Z := & \left(\partial F_{g,b} \times I \times I \right) \cup \left(P \times I\right)\\
	\end{array}$$
called the \emph{inner and outer boundaries of} $Z.$ See Figure~\ref{F:Decompose}. $\partial_{In} Z$ is the portion of $\partial Z$ which gets  glued to the other pieces in the trisection, whereas $\partial_{Out} Z$ contributes to $\partial X.$

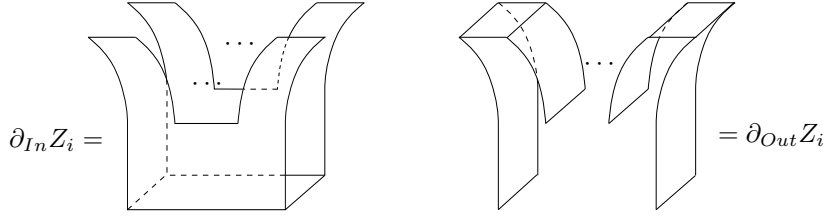
\begin{figure}[ht!]
\centering{
	\labellist
		\pinlabel $\partial_{In}Z_i=$ [r] at 20 60
		\pinlabel $=\partial_{Out}Z_i$ [l] at 533 65
		
		\pinlabel $\cdots$ at 440 125
		\pinlabel $\cdots$ at 104 108
		\pinlabel $\cdots$ at 132 142 
	\endlabellist
	\includegraphics[scale=.45]{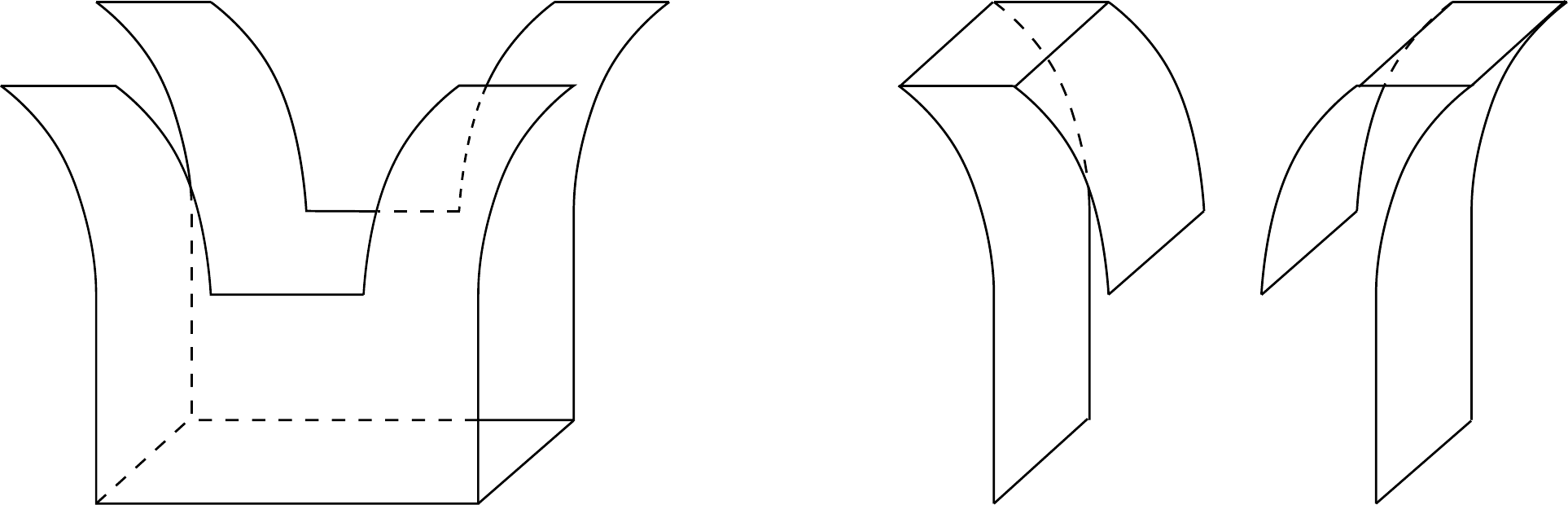}
}
	\caption{Decomposing $\partial Z$ \label{F:Decompose}}
\end{figure}

There is a standard \emph{generalized Heegaard splitting} of $\partial_{In} Z,$ i.e., a decomposition of a $3$--manifold with boundary $M = C_1 \cup C_2,$ where $C_1 \cong C_2$ are compression bodies which intersect along a surface with boundary. We decompose $\partial_{In} Z$ as
	$$\partial_{In} Z = \left( C \times \{0\} \cup F_{g,b}\times [0,1/2]\right) \cup \left( F_{g,b} \times [1/2,1] \cup C \times \{1\} \right).$$
which we will denote as  $\partial_{In} Z= Y_0^+ \cup Y_0^-,$ where $Y_0^+ \cap Y_0^- \cong F_{g_0,b}.$ We also allow for further stabilizations of this splitting (on the interior of the surface) some number of times (possibly zero) which increases the genus of the splitting. We will denote this stabilized, standard splitting as $\partial_{In} Z = Y^+ \cup Y^-,$ where $Y^+ \cap Y^- = F_{g,b}.$ It should be noted that the stabilizations involved do not alter the $4$--manifold $Z_i$ in any way, only the decomposition of the $3$--manifold $\partial_{In} Z.$

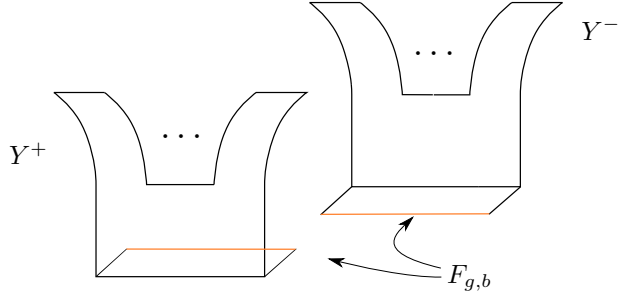
\begin{figure}[ht!]
\centering{
	\labellist
		\pinlabel {$Y^+$} at -20 100
		\pinlabel {$Y^-$} at 440 200
		\pinlabel {$F_{g,b}$} [l] at 310 0
		
		\pinlabel \resizebox{!}{6pt}{$\cdots$} at 105 113
		\pinlabel \resizebox{!}{6pt}{$\cdots$} at 308 180
	\endlabellist
	\includegraphics[scale=.47]{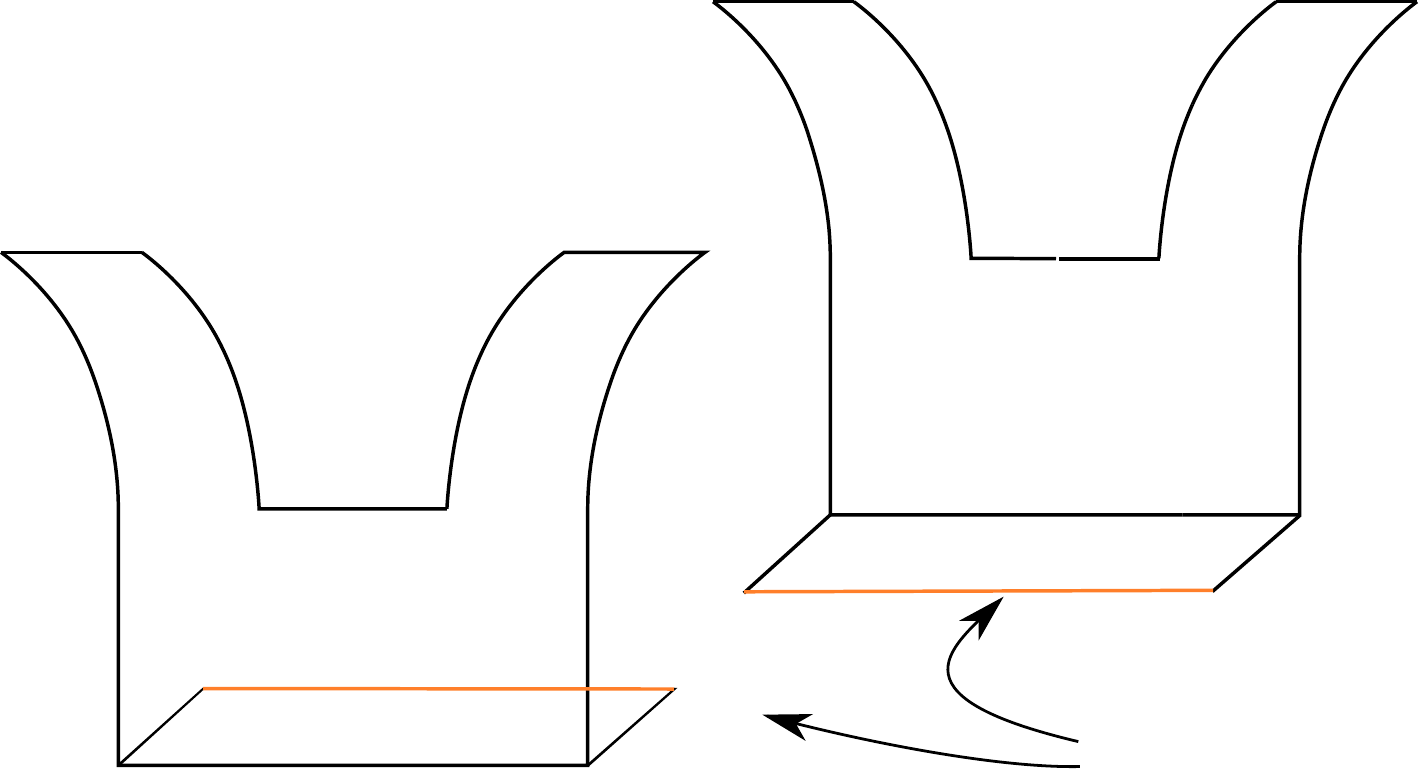}
}
	\label{F:innerboundary}
	\caption{Decomposing $\partial_{In} Z$ as $Y^+ \cup Y^-$}
\end{figure}

The ``nice intersections'' mentioned earlier can now be defined: $X_i \cap X_{i+1} \cong Y^+$ and $X_i \cap X_{i-1} \cong Y^-.$ Alternately phrased, $(X_i \cap X_{i+1}) \cup (X_i \cap X_{i-1})$ is this particular generalized Heegaard splitting of $\partial_{In} X_i \cong \partial_{In} Z.$ We now give the proper definition using the above notation

\begin{definition}
A \emph{relative trisection} of a smooth $4$--manifold with boundary $X$ is a decomposition $X = X_1 \cup X_2 \cup X_3$ such that, for some $Z$ with splitting $\partial_{In} Z = Y^+ \cup Y^-$ constructed as above
	\begin{enumerate}[i)]
		\item for each $i$ there exists a diffeomorphism $\varphi_i: X_i \rightarrow Z,$
		
		\item for each $i,$ we have $\varphi_i(X_i \cap X_{i+1}) = Y^+$ and $\varphi_i(X_i \cap X_{i-1}) = Y^-$
	\end{enumerate}
where indices are taken mod $3$. We will sometimes denote a trisection of $X$ as $\tri_X,$ or $\tri.$
\end{definition}

It should be noted that the phrase ``for some $Z$'' hides many quantifiers which are necessary in defining relative trisections. We omit them in the definition because, in the case of multiple boundary components, the notation for a relative trisection can quickly become messy. Thus, we simply refer to a relative trisection as $\tri$ with the understanding that the topology of $Z$ is determined by:
\begin{enumerate}[i)]
	\item $m = |\partial X|$ 
	\item $n$ - the number of $2$--handles attached
	\item $g$ - the genus of the trisection surface
	\item $b$ - the number of boundary components of the trisection surface
	\item $p_i$ - the genus of each component $P_i$ of $P$
\end{enumerate}
As a consequence, the triple intersection $X_1 \cap X_2 \cap X_3= F_{g,b}$ is a surface with boundary called the \emph{trisection surface}, and the outer boundaries comprise $\partial X.$ Let us denote $\partial_{Out} X_i = \varphi^{-1}_i(\partial_{Out} Z).$ Notice that $\partial_{Out} X_i = X_i \cap \partial X.$ The connected components of $\partial_{Out} X_i$ are given by $P_i \times I$ together with $\nu \partial P_i,$ a $3$--dimensional neighborhood of $\partial P_i.$ Thus, gluing the $X_i$'s to one another induces a fibration $\partial X \setminus \nu\partial P \rightarrow S^1$ with fiber $P.$ In other words, $(P, \phi)$ is an abstract open book corresponding to $\partial X,$ where $\phi$ is determined by the attaching maps $\{\varphi_i\}.$

We have thus proved the following lemma, which generalizes Gay and Kirby's~\cite{gay} result to smooth, compact $4$--manifolds with an arbitrary number of boundary components. 
\begin{lemma}\label{L:inducedobd}
	A relative trisection of $X$ induces an open book decomposition of each component of $\partial X.$
\end{lemma}

\begin{example}[Relative Trisection of $B^4$]
The simplest relative trisection is the trivial trisection of $B^4.$ Let us use the identification $B^4 \cong D^2 \times D^2,$ where $D^2 =\{re^{i\theta} \in \mathbb{C}| r\leq 1\}.$ Decompose the unit disk $D^2 = D_1 \cup D_2 \cup D_3,$ where $D_j = \{re^{i\theta} \in D^2| 2\pi j/3\leq \theta \leq 2\pi(j+1)/3\}$, and let $p_1:D^2 \times D^2 \rightarrow D^2$ be the projection onto the first factor. Defining $X_j = p_1^{-1}(D_j)$ yields a relative trisection of $B^4$ where
\begin{itemize}
	\item[-] $X_i \diffeo B^4$

	\item[-] $X_i \cap X_{i+1} \diffeo B^3$

	\item[-] $X_1 \cap X_2 \cap X_3 \diffeo D^2$
\end{itemize}
Figure~\ref{F:B4} is a $3$--dimensional representation of this trisection, where the dimension of each sphere or disk is one less than the dimension of the sphere or disk it represents. 
	{\centering\begin{figure}[ht!]
		\centering{
			\includegraphics[scale=.5]{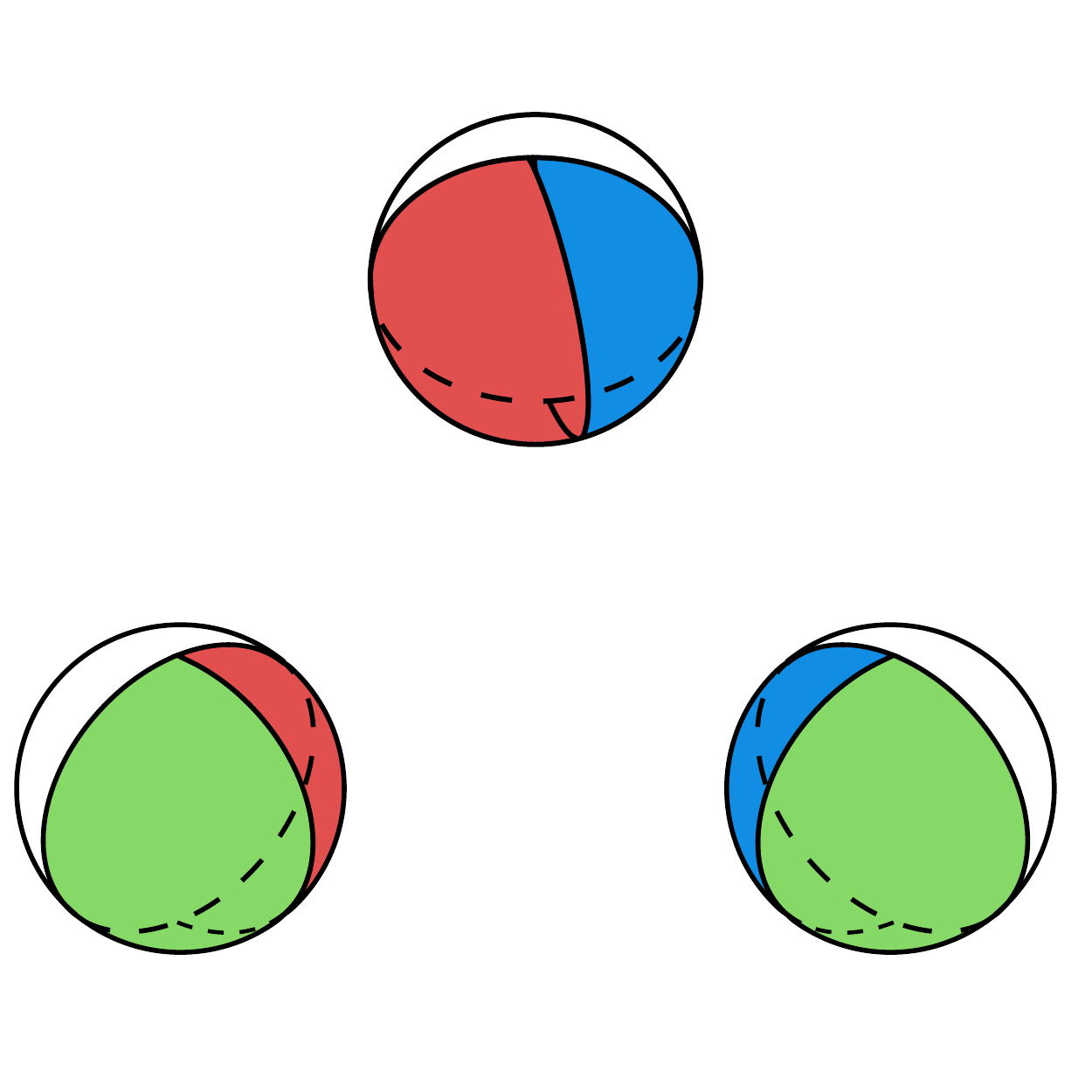}
			
			\caption{Trivial Trisection $\mathcal{B}$ of $B^4$ \label{F:B4}}
			}
	\end{figure}}
The colored regions on a given $X_i$ comprise $\partial_{In} X_i$ (which are modeled by $D^2$). We then take a genus--$0$ generalized Heegaard splitting of $\partial_{In} X_i \diffeo B_i^{3 +} \cup_{D^2} B_i^{3 +}.$ Each $B_i^{3 \pm}$ is colored so as to indicate where $X_i$ will glue to $X_j.$ Taking indices mod 3, we trivially glue $B_i^{3 +}$ to $B_{i+1}^{3 -}.$ Doing so yields $B^4 \diffeo X_1 \cup X_2 \cup X_3.$ Moreover, we see that $\partial_{Out} X_i \diffeo B^3$ and our gluing gives us 
	$$\partial X = \bigcup_i \partial_{Out} X_i \diffeo S^3.$$
Notice that the triple intersection has boundary. In Figure~\ref{F:B4}, it is represented by the arc ($B^1$) which separates each color on the inner boundaries. As one might expect, the induced open book $\trib|_{\partial B^4}$ is the trivial open book on $S^3.$
\end{example}

\begin{example}[Relative Trisection of $S^3 \times I$]
Let $\trib_0$ be the trivial trisection of $B^4$ and let us take the connected sum $(B^4, \trib_0) \# (B^4, \trib_0)$ in such a way that neighborhoods of points in the trisection surfaces are identified. (For clarity, let us denote the pieces of the trisections as $Z_1 \cup Z_2 \cup Z_3$ and $Z'_1 \cup Z'_2 \cup Z'_3$, where $Z_i \cong Z'_i$.) The claim is that this connected sum $\trib_0 \# \trib_0$ defines a relative trisection of $B^4 \# B^4 \cong S^3 \times I.$
	Let $B_i$ be a $3$--ball whose boundary is decomposed as the union of two disks along their boundaries $\partial B_i = D^{+}_i \underset{\partial}{\cup} D^{-}_i,$ where $D^{\pm}_i \cong D^2.$ By making the identifications $D^{+}_i \sim D^{-}_{i+1},$ we obtain $S^3 = B_1 \cup B_2 \cup B_3$. We then extend this decomposition to $S^3 \times I = (B_1 \times I) \cup (B_2 \times I)\cup (B_3 \times I).$
	Notice that this is not a relative trisection, as each $B_i \times I$ has not been given the structure of a model piece. In particular, it is not a thickening of a compression body from an annulus to two disjoint disks. However, each $B_i \times I$ serves as a $1$--handle joining $Z_i$ and $Z'_i$ in such a way that 
	$$(B_1 \cap B_2 \cap B_3) \times \{0\} \subset Z_1 \cap Z_2 \cap Z_3$$
	and
	$$(B_1 \cap B_2 \cap B_3) \times \{1\} \subset Z'_1 \cap Z'_2 \cap Z'_3.$$
	Thus, we have a decomposition
	$$S^3 \times I = (Z_1 \underset{1-h}{\cup} Z'_1) \cup (Z_2 \underset{1-h}{\cup} Z'_2) \cup (Z_3 \underset{1-h}{\cup} Z'_3).$$
Each $Z_i \underset{1-h}{\cup} Z'_i$ is diffeomorphic to $C \times I$, where $C$ is the compression body obtained from attaching $3$--dimensional $2$--handle to $S^1 \times I \times I$ along $S^1 \times \{1/2\}\times \{1\}$. Moreover, we have that $(Z_i \underset{1-h}{\cup} Z'_i) \cap (Z_{i+1} \underset{1-h}{\cup} Z'_{i+1}) \cong B^3.$ Since we have taken the connected sum along the interiors, the boundary data has not been altered. Therefore, $\trib_0 \# \trib_0$ is a trisection of $S^3 \times I$ such that 
\begin{itemize}
	\item[-]each piece of the trisection is $B^3$,
	\item[-]the trisection surface is $S^1 \times I$,
	\item[-]each boundary component is endowed with the trivial open book decomposition of $S^3.$
\end{itemize}
	
\begin{figure}[ht!]
\centering{
	\includegraphics[scale=.6]{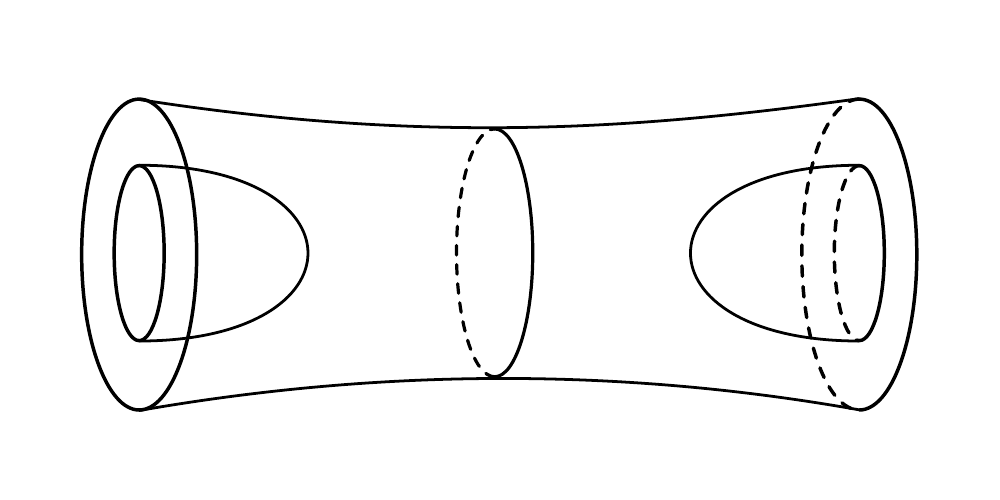}
	
	\label{F:cobordismex}
	\caption{Compression body $C\cong X_i\cap X_j$ for $S^3 \times I$}
}
\end{figure}
\end{example}
The above example exhibits a bit more than a trisection of $S^3 \times I.$ In fact, it shows that a connected sum of any two trisections, relative or closed, results in a trisection.

\begin{theorem}\label{T:4dexistence}
	Let $X$ be a smooth, compact, connected $4$--manifold with boundary such that each connected component of $\partial X$ is equipped with a fixed open book decomposition. There exists a trisection of $X$ which restricts to $\partial X$ as the given open books.
\end{theorem}
The following proof is a natural extension of the proof given by Gay and Kirby in \cite{gay} to the case of multiple boundary components.
\begin{proof}
	Let $(B, \phi)$ be an open book decomposition of $\partial X$ with page $P.$ If $\partial X$ has $m$ connected components, then so does $P.$ We will use the given boundary data to construct a Morse function $f: X \rightarrow I.$
	
	Extend $\phi: \partial X \setminus \nu B \rightarrow S^1$ to the whole of $\partial X$, $\phi: \partial X \rightarrow D^2$ by $(x,z) \mapsto z$ for every $(x,z) \in B \times D^2.$ Then fix an identification of $D^2$ with $I \times I$ and compose $\phi$ with the projection onto the first factor, giving us a smooth map $f:\partial X \rightarrow I$ such that
	\begin{enumerate}[i)]
		\item $f^{-1}(0) \cong \underset{i = 1}{\overset{m}{\bigsqcup}}(P_i \times I) \cong f^{-1}(1)$
		\item $f^{-1}(t) \cong  \underset{i = 1}{\overset{m}{\bigsqcup}}\left(P_i \times\{0\} \cup (\partial P_i \times I) \cup P_i \times \{1\} \right) \; 0 < t < 1.$
	\end{enumerate}
Extend $f$ to a Morse function on all of $X$ and consider the handle decomposition given by $f$. Notice that since $X$ is connected, if $m>1$ then such a function necessarily admits $1$--handles. Let $h_i$ denote the number of $i$--handles. Without loss of generality, we can assume that the handles are ordered by index. Moreover, by adding canceling pairs we can arrange for $h_1 = h_3.$

Let $\varepsilon, a \in (0,1)$ be such that $f^{-1}([0, \varepsilon]) = f^{-1}(0) \times I$ and $(\varepsilon, a)$ contains all of the index $1$ critical values, but no others. We then have 
	$$\begin{array}{ll}
		f^{-1}([0, \varepsilon])& \cong f^{-1}(0) \times I\cr
							& \cong (P \times I) \times I\cr
							& \cong \underset{i = 1}{\overset{m}{\bigsqcup}}\left(\natural^{l_i}S^1 \times B^3\right),\\
	\end{array}$$
where $l_i = 2p_i + (b_i - 1).$ Define $X_1 = f^{-1}([0,a])$. Since $X$ is connected, we have
	$$\begin{array}{ll}
		X_1& \cong f^{-1}([0, \varepsilon]) \cup 1\textrm{-handles}\cr
			& \cong \underset{i = 1}{\overset{m}{\natural}}\left((\natural^{l_i}S^1 \times B^3) \natural (\natural^{h_1 -m+1}S^1 \times B^3)\right)\cr
			& \cong \natural^{l+h_1 - m+1}S^1 \times B^3,\\
	\end{array}$$
where $l = \underset{i}{\Sigma} l_i.$

We will now give $f^{-1}(a)$ a Heegaard splitting: Define $N = f^{-1}(\varepsilon)$ and $M = \overline{\partial(f^{-1}([0, \varepsilon]))\setminus N}.$ It is not hard to see that $M \cong N \cong P \times I,$ and thus there is a natural generalized Heegaard splitting of $N \cong \left(P \times [0, 1/2]\right) \cup \left(P \times [1/2,1]\right).$ Thus when we attach the $1$--handles, some of which connect components to each other, we have a sort of ``unbalanced decomposition'' $\partial X_1 = M \cup N',$ where $N'$ is diffeomorphic to
	$$\left(\left(F_{p,b} \times [0, 1/2]\right) \natural H_1\right) \bigcup \left(\left(F_{p,b} \times [1/2,1]\right) \natural H_2\right)$$
and $\#^{h_1 - m}S^1 \times S^2 = H_1 \cup H_2$ is the standard genus $h_1 - m$ Heegaard splitting. Note that $N' \cong f^{-1}(a).$ Let us denote this generalized Heegaard splitting $N' = Y^+ \cup Y^-,$ where $Y^+$ and $Y^-$ are compression bodies from $F_{p+h_1 -m+1, b}$ to $P$ which intersect along the surface of greater genus.

Let $L \subset N'$ be the framed link which corresponds to the attaching spheres of the $2$--handles given by $f.$ Project $L$ onto the splitting surface $F_{g,b} = Y^+ \cap Y^-$ in such a way that each component of $L$ non-trivially contributes to the total number of self intersections, or crossings, denoted by $c.$ This can be done by Reidermeister $1$ moves if necessary. 

We first consider the special case where $c = h_2.$ Stabilizing the generalized Heegaard splitting $Y^+ \cup Y^-$ at every crossing of $L$ resolves the double points by providing $1$--handles whose co-cores intersect a unique link component exactly once. Additionally, every link component has such an intersection by construction. We also wish for the framings of the now embedded link $L$ to be consistent with the page framings. This is achieved by adding a kink via a Reidermeister $1$ move and stabilizing at the new crossing. This changes the page framing by $\pm1,$ which will allow us to achieve any framing through this process. Although we may have stabilized several times, we still denote $N' = Y^+ \cup Y^-$ with $Y^+ \cap Y^- = F_{g,b}.$ This is pictured in Figure~\ref{F:Kink}.

\begin{figure}[ht!]
\centering{
	\labellist
		\pinlabel \resizebox{16pt}{4pt}{$\rightarrow$} at 132 28
		\pinlabel \resizebox{!}{6pt}{R1} at 132 37
		
		\pinlabel \resizebox{18pt}{4pt}{$\rightarrow$} at 281 28
		\pinlabel \resizebox{!}{6pt}{stabilize} at 280 37
	\endlabellist
\includegraphics[scale=.8]{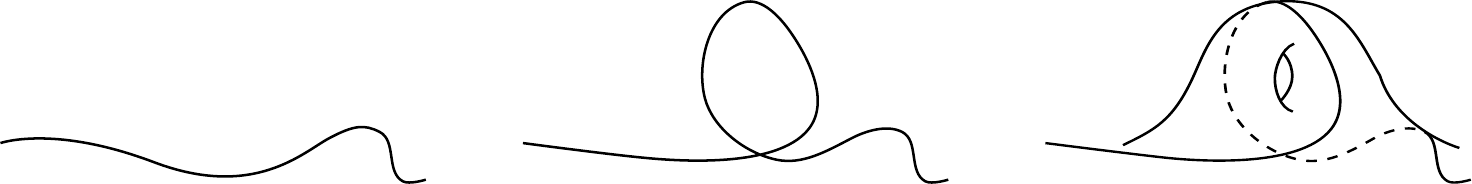}
}

\caption{Changing the Page Framing by $\pm1$\label{F:Kink}}
\end{figure}
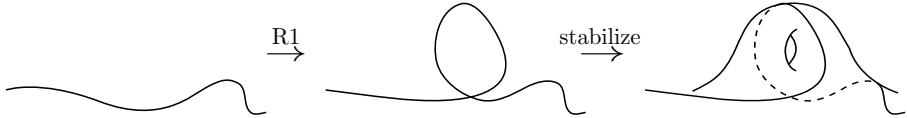

Let us now define $X_2$ to be a collar neighborhood of $Y^+$ in the complement of $X_1$ together with the $2$--handles of $X.$ It is important that in $X_2,$ the $3$--dimensional $1$--handles of $Y^+$ give rise to $4$--dimensional $1$--handles of $X_2.$ Since we have just arranged for the attaching sphere of each $2$--handle to intersect the co-cores of the $1$--handles, there are $c=h_2$ canceling $1-2$ pairs in $X_2.$ (We can slide $1$--handles over one another to obtain a one-to-one correspondence between $2$--handles and co-cores of $1$--handles).

We can now verify that $X_2$ is a handlebody of the appropriate genus. First, we have that 
	$$\begin{array}{rl}
		Y^+	& \cong \left(F_{p, b} \times I\right) \natural H_1\cr
			& \cong \left(\natural^l S^1 \times D^2\right) \natural \left(\natural^{h_1 -m+1+c} S^1 \times D^2\right)\cr
			& \cong \natural^{l+h_1 -m+1+c} S^1 \times D^2,\\
	\end{array}$$
where we have taken the view that the $c$ stabilizations occur in $H_1 \cup H_2.$ Finally, since we have arranged for the $2$-handles to cancel $1$--handles, we obtain the desired result:
	$$\begin{array}{rl}
		X_2	& \cong Y^+ \times [0,\varepsilon] \underset{L}{\cup} 2-\textrm{handles}\cr
			& \cong \left(\natural^{l+h_1 -m+1+c} S^1 \times B^3\right) \underset{L}{\cup} 2-\textrm{handles}\cr
			& \cong \natural^{l+h_1 -m+1} S^1 \times B^3,\\
	\end{array}$$

Finally, define $X_3 := \overline{X\setminus (X_1 \cup X_2)}.$ Since $h_1=h_3$, ``standing on your head'' gives us $X_1 \cong X_3.$

As for the intersections, $X_1 \cap X_2 \cong Y^+$ and $X_1 \cap X_3 \cong Y^-$ by definition. To see $X_2 \cap X_3 \cong Y^+$ we exploit the one-to-one correspondence between link components and a subset of the co-cores of $1$--handles of $Y^+.$ Each surgery on $Y^+ \times \{1\}$ defined by a link component of $L\subset int(Y^+\times \{1\})$ can be done in a unique $S^1 \times D^2$ summand of $Y^+ \times \{1\} \cong \natural S^1 \times D^2.$ Such a surgery on $S^1 \times D^2$ results in $S^1 \times D^2$ and simply changes which curve bounds a disk. Thus, the surgery $3$--manifold $(Y^+ \times \{1\})_L$ is diffeomorphic to $Y^+.$ This completes the proof when $c=h_2.$

In the general case when $c > h_2,$ we add $c-h_2$ cancelling $1-2$ pairs and $c-h_2$ canceling $2-3$ pairs in the original handle decomposition of $X$ given by $f.$ After said perturbation of $f$, we modify the pieces accordingly. (Some modifications are required to make the pieces of the trisection diffeomorphic to each other. Other modifications are needed so that the attaching spheres of the $2$--handles are embedded in the trisections srurface.) We have $X'_1 = X_1 \natural^{c-h_2} S^1 \times B^3,$ whose boundary is similarly decomposed as $\partial X'_1 = M \cup \left(N' \#^{c-h_2}S^1 \times S^2\right).$ Additionally, we have a new generalized Heegaard splitting of $N' \#^{c-h_2}S^1 \times S^2$
	$$(Y^+ \#\mathcal{H}_1) \cup (Y^- \# \mathcal{H}_2),$$
where $\mathcal{H}_1 \cup \mathcal{H}_2$ is the standard genus $c-h_2$ Heegaard splitting of $\#^{c-h_2} S^1 \times S^2.$ That is, we have stabilized $Y^+ \cup Y^-$ $c-h_2$ times, once for each newly added $1$--handle. The new generalized Heegaard surface $F$ is of genus $p+h_1 - m+c - h_2$ and has $b$ boundary components. Moreover, the original link $L$ projects onto $F$ as it did before. However, we now have an additional $2(c - h_2)$ link components corresponding to the newly added $2$--handles. The half which correspond to the $1-2$ pair necessarily have the canceling intersection property discussed above. The half corresponding to the $2-3$ pairs project onto to $F$ as a $0$ framed unlink which bounds disks in $F.$ Stabilizing $Y^+ \underset{F}{\cup} Y^-$ for the last time(s) near each unknot allow us to slide the links into canceling position with the new $S^1 \times S^1$ summands of $F.$
\end{proof}

\end{section}

\begin{section}{The Gluing Theorem}\label{S:gluingtheorem}

We now restate the gluing theorem.
\begin{theorem}\label{T:gluing}
Let $X$ and $W$ be smooth, compact, connected oriented $4$--manifolds with non-empty boundary equipped with relative trisections $\tri_X$ and $\tri_W$ respectively. Let $B_X \subset \partial X$ be any collection of boundary components of $X$ and $f:B_X \hookrightarrow \partial W$ an injective, smooth map which respects the induced open book decompositions $\tri_X|_{B_X}$ and $\tri_W|_{f(B_Z)}.$ Then $f$ induces a trisection $\tri = \tri_X \underset{f}{\cup} \tri_W$ on $X \underset{f}{\cup} W.$
\end{theorem}
Note that if $B_X = \partial X = f^{-1}(\partial W)$, then the induced trisection is of a closed $4$--manifold. Otherwise, we have a relative trisection of a $4$--manifold with boundary. 
Schematics for the two possible gluings are given in Figure~\ref{F:gluings}. Note that the schematic on the right depicts the gluing of only one boundary component from each manifold, but should be thought of as ``not all components get glued.''

{\centering\begin{figure}[ht!]
\centering{
	\labellist
		\pinlabel {\rotatebox{-90}{\resizebox{16pt}{4pt}{$\leftrightarrow$}}} at 18 174
		\pinlabel $\cdots$ at 100 105
		
		\pinlabel {\rotatebox{90}{\resizebox{16pt}{4pt}{$\leftrightarrow$}}} at 182 174
		\pinlabel $\cdots$ at 98 232
		
		\pinlabel {\rotatebox{-90}{\resizebox{16pt}{4pt}{$\leftrightarrow$}}} at 509 174
		\pinlabel $\cdots$ at 594 108
		\pinlabel $\cdots$ at 434 238
		
		\pinlabel \resizebox{8pt}{!}{$Q_1$} at -15 205
		\pinlabel \resizebox{8pt}{!}{$P_1$} at -15 143
		\pinlabel \resizebox{11pt}{!}{$P_m$} at 225 143
		\pinlabel \resizebox{11pt}{!}{$Q_m$} at 225 205
		
		\pinlabel \resizebox{8pt}{!}{$Q_i$} at 564 205
		\pinlabel \resizebox{8pt}{!}{$P_i$} at 464 143
		
	\endlabellist
\includegraphics[scale=.4]{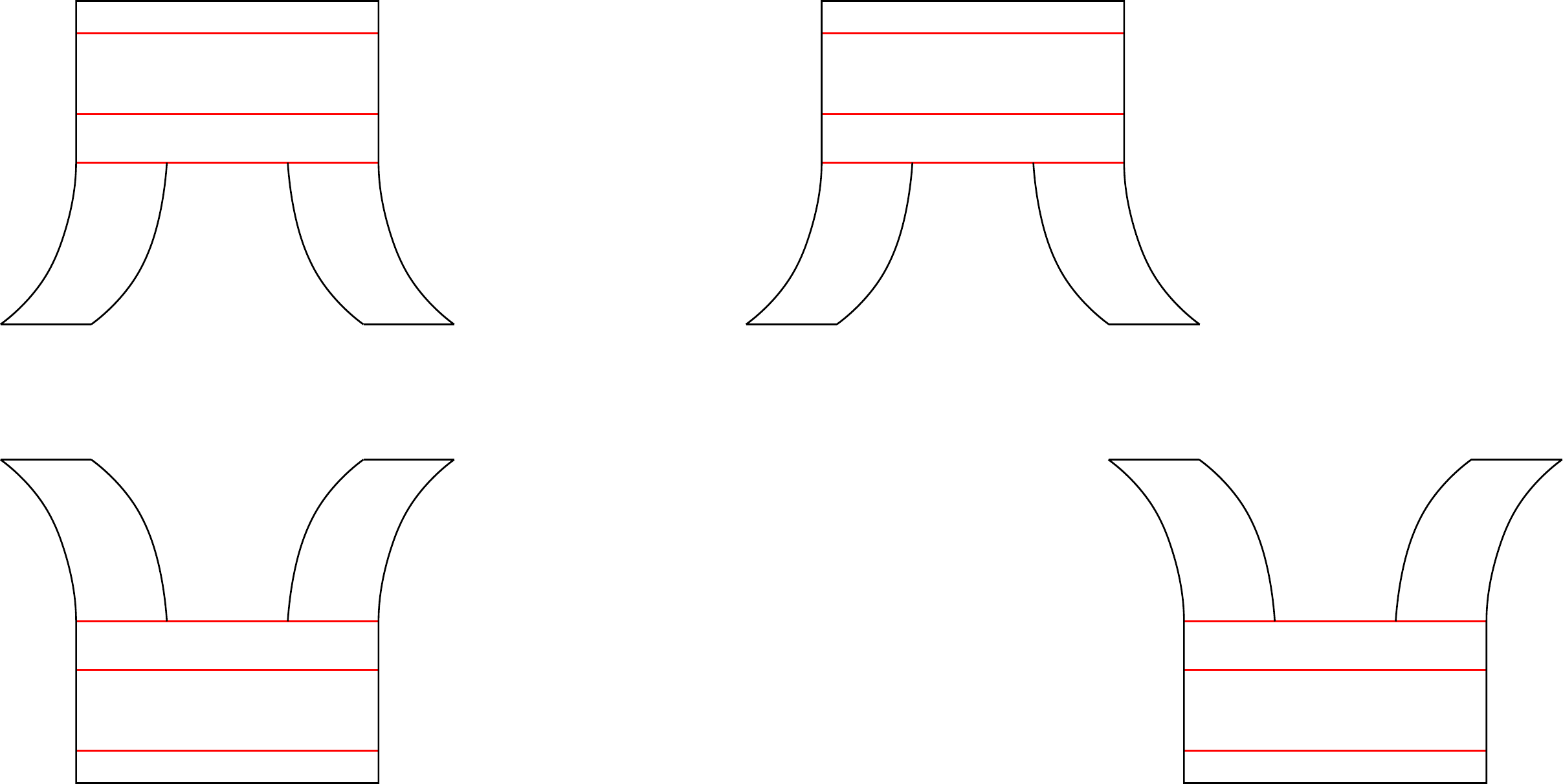}
}
	\caption{Gluing Relative Trisections \label{F:gluings}}
\end{figure}
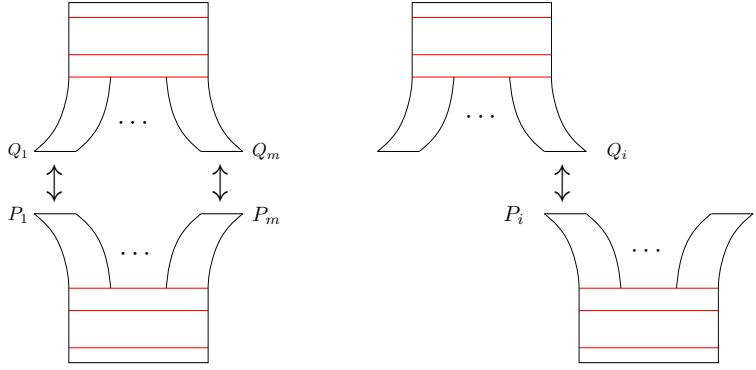}

\begin{proof}
Let $P = \underset{j = 1}{\overset{m}{\sqcup}} P_j$ and $Q = \underset{j = 1}{\overset{\mu}{\sqcup}} Q_j$ be the pages of the open books induced by $\tri_X$ and $\tri_W$ respectively, where $P_j \cong F_{p_j, b_j}$ and $Q_j \cong F_{q_j, d_j}$ for each $j.$ Additionally, let $C$ and $B$ denote the compression bodies which give us the $X_i$'s and $W_i$'s (i.e., $C \times I \cong X_i$ and $B \times I \cong W_i$). Let $n$ and $\eta$ denote the number of $3$--dimensional $1$--handles in the constructions of $C$ and $B$ respectively.
	
We begin with the case $D =\partial X \cong \partial W$ (and thus $m=\mu$). Our gluing is defined in the natural way, by attaching $\partial_{Out} X_i$ to $\partial_{Out}W_i$ via $f.$ Our new pieces are given by
	$$Z_i := X_i \underset{f}{\cup} W_i  \;,$$
where $f(x)=x$ for all $x\in \partial_{Out} X_i.$ We wish to show that $Z_i$ is a $4$--dimensional handlebody, and that $Z_i \cap Z_{i+1}$ and $Z_i \cap Z_{i-1}$ are $3$--dimensional handlebodies. Since $X_i$ and $W_i$ are thickened compression bodies, we will reduce these to $3$--dimensional arguments.

To see $X_i \cap W_i$ is a handlebody, we require the following simple fact:

\begin{lemma}\label{L:sidefold}
Define the quotient space $M=F_{g,b} \times I / (x,t) = (x, 1-t)$ for every $x \in \partial F_{g,b}$ and every $t\in I.$ $M$  is diffeomorphic to $F_{g,b} \times I.$
\end{lemma}
The main idea of the proof lies in Figure~\ref{F:sidefold}. We leave the details to the reader. 

\begin{centering}
\begin{figure}[h!]
\centering{
	\labellist
		\pinlabel {$F_{g,b} \times I$} [l] at 368 419
		\pinlabel {$\overset{\textit{fold}}{\longrightarrow}$} at 189 234
		\pinlabel {$M$} at 342 72
	\endlabellist
	\includegraphics[scale=.4]{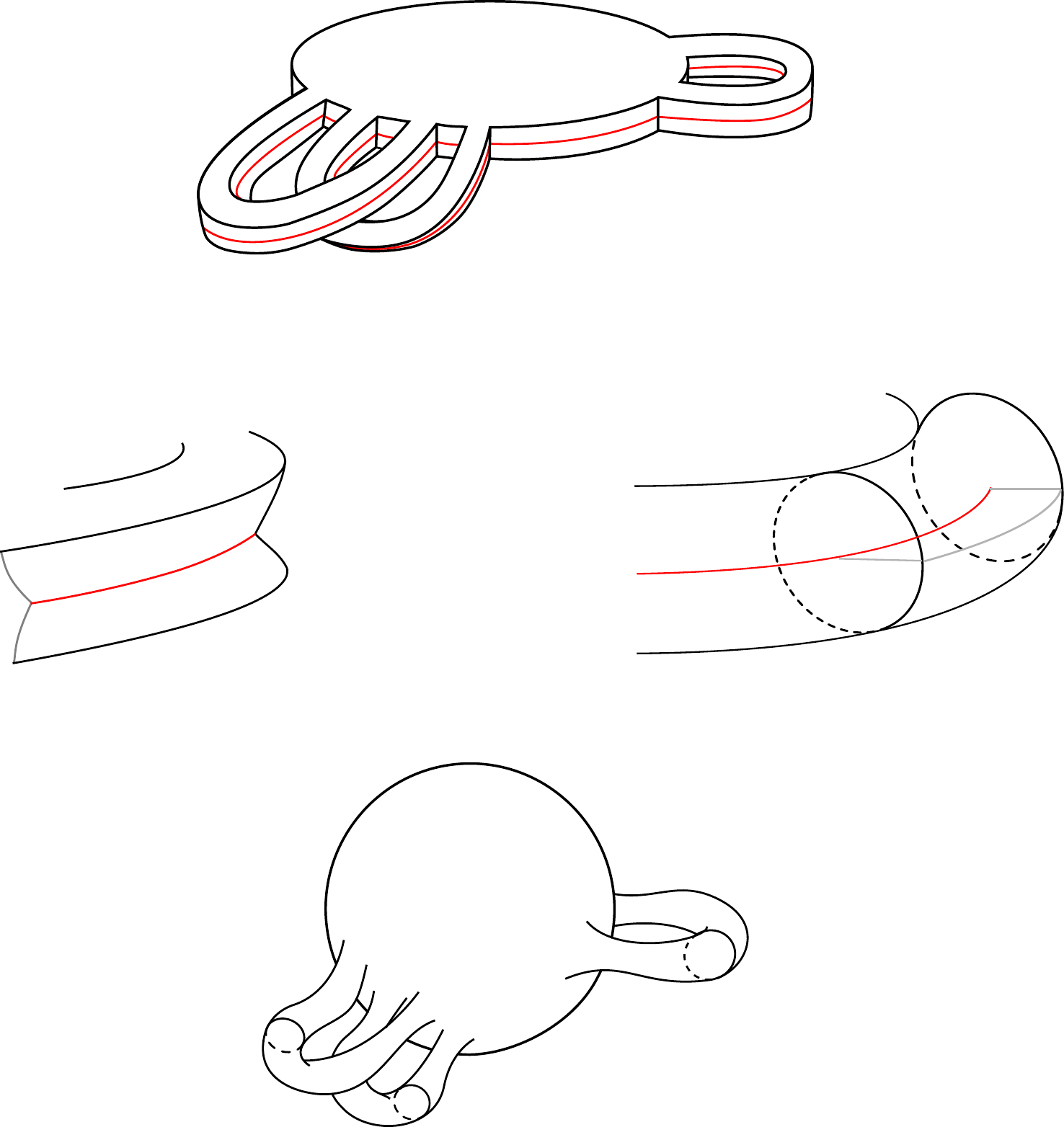}
}
	\caption{Proof of Lemma~\ref{L:sidefold} \label{F:sidefold}}
\end{figure}\end{centering}

Lemma~\ref{L:sidefold} gives us $(P_j\times I) \underset{f}{\cup} (Q_j\times I)\cong \natural^{l_j} S^1 \times D^2.$ We will finish constructing $A:= C \cup B/ \sim$ by attaching the $1$--handles of $B$ and $C$ to $(P_j\times I) \underset{f}{\cup} (Q_j\times I)$. Since the $1$--handles in the construction of $B$ and $C$ are attached along the interior of level sets, the gluing and the attachment of the $1$--handles can be done independently. The $m-1$ $1$--handles coming from $C$ connect the $\natural^{l_j} S^1 \times D^2$ to one another, yielding $\natural^l S^1 \times D^2.$ The $m-1$ $1$--handles coming from $B$ (which connected the $Q_i$'s in $B$) now increase the genus of the handlebody, leaving us with $\natural^{l+m-1}S^1\times D^2.$ We then proceed to attach the remaining $n-(m-1)$ $1$--handles from $C$ and the remaining $\eta-(m-1)$ $1$--handles from $B$. Thus, $A\cong \natural^k S^1\times D^2$, where $k=l+n+\eta-(m-1).$ By definition of our gluing, we have $A = \partial_{In} X_i \cup \partial_{In}W_i/ \sim.$ Thus, $A = Z_i \cap Z_j$ is a $3$--dimensional handlebody of genus $k = l + n-(m-1) + \eta -(\mu -1).$ Noting that $Z_i = A \times I$ gives us the desired result.

The more difficult case is when we wish to result in a relative trisection. For simplicity, we will prove this case when gluing $X$ and $W$ along a single boundary component given by a map which takes $P_1$ to $Q_1$ as in the right side of Figure~\ref{F:gluings}. The argument easily generalizes to multiple boundary components.

Let us view $B$ and $C$ as being constructed in reverse as mentioned in the previous section. Again, the fact that the $1$--handles in these constructions are attached to the interiors of $P$ and $Q$ allows us to glue $P_1 \times I$ to $Q_1\times I$ before connecting components of the compression bodies. In other words, if we denote $M = (P_1 \times I \underset{f}{\cup} Q_1 \times I),$ then $A = C \cup B / \sim$ can be constructed by attaching $1$--handles to 
	$$\left(\underset{i=2}{\overset{m}{\bigsqcup}}(P_i \times I)\right) \sqcup M \sqcup \left(\underset{i=2}{\overset{\mu}{\bigsqcup}}(Q_i \times I)\right).$$
Lemma~\ref{L:sidefold} again gives us $M \cong \natural^{l_1}S^1\times D^2$ which can be constructed by attaching $l_1$ $3$--dimensional $1$--handles to $B^3.$ Thus $A$ can be constructed as follows: Attach $m-2$ $1$--handles to $P\times I$ and $\mu - 2$ $1$--handles to $Q\times I$ so that each are connected. We then attach these components to $B^3$ (the $0$-handle of M). Note that these two $1$--handles giving us a connected manifold are the $1$--handles which connect $P_1 \times I$ and $Q_1 \times I$ to the remaining thickened pages in $P\times I$ and $Q\times I$ respectively and do not increase the genus. To complete the construction, we attach $l_1$ $1$--handles, coming from the construction of $M.$ This, gives us a compression body $A$ whose ``smaller genus'' end (pages of open book) is $\underset{j=2}{\overset{m}{\cup}}P_j \cup \underset{j=2}{\overset{\mu}{\cup}}Q_j$ and ``larger genus'' end is a surface of genus
	\begin{align}\label{eq:gluedgenus}
	\sum_{j=2}^mp_j + \sum_{j=2}^{\mu}q_j + (n-m+1) +(\eta-\mu+1) + \left(2p_1 + b_1-1\right)
	\end{align}
with $(b-b_1) + (d-d_1)$ boundary components.

Although the new trisection genus given by (\ref{eq:gluedgenus}) is quite involved, the idea behind the calculation is quite simple. If $\tri_X$ and $\tri_W$ have relative trisection surfaces $F_X$ and $F_W$ respictively, we obtain the new trisection surface $F_Z$ by identifying the boundaries of $F_X$ and $F_W$ corresponding to the open books $(P_1, \phi_1)$ and $(Q_1, \psi_1)$ as prescribed by $f$.

When gluing trisections together along $s>1$ boundary components, we simply modify our calculation of the genus of $A$ to account for the fact that $s-1$ of the $1$--handles coming from $B$ now increase the number of $S^1 \times D^2$ summands. In general, we have $k = \underset{i=1}{\overset{s}{\Sigma}}l_{j_i} + n - (m-1) + \eta - (\mu - 1) +(s - 1).$
\end{proof}
\end{section}


\begin{section}{Relative Stabilizations\label{S:relstab}}
In this section we describe a stabilization of a relative trisections which is significantly different than that given by Gay and Kirby in that it changes the boundary data of a single boundary component and increases the trisection genus by either one or two. The effect such a \emph{relative stabilization} has on the open book decomposition of the chosen boundary component of $\partial X$ is a Hopf stabilization; both positive and negative stabilizations can be achieved.

Given a relative trisection $\tri$ of $X$, consider a corresponding function $f:X \rightarrow D^2$ as constructed in Theorem~\ref{T:4dexistence} (without identifying $D^2$ with $I\times I$ and projecting onto the first factor). We begin by introducing a Lefschetz singularity as in Section~\ref{S:Background}. In the case of multiple boundary components, we must choose the boundary component to which we attach the one handle, taking care that the attaching sphere of the canceling $1$--handle is contained in a single boundary component $M_i \subset \partial X$ with open book $(P_i, \phi_i).$ (Otherwise, we would be changing our $4$--manifold by connecting boundary components.) We attach the $2$--handle just as before, in the neighborhood of a regular value $y_0$, creating a singularity locally modeled by $(u,v)\mapsto u^2 + v^2.$ The left half of Figure~\ref{F:wrinkle} shows a neighborhood of the singularity and a neighborhood of the vanishing cycle.

\begin{remark}
Notice that the attaching spheres of the $1$--handle can be attached to the same binding component or to different binding components. We discuss this difference below.
\end{remark}

Let $Z_f \subset D^2$ denote the original critical values of $f$ before introducing the canceling pair. This is a codimension $1$ set which is given by indefinite folds with finitely many cusps and crossings. We wish to show that we can ``move $x_0$ past'' all but finitely many points of $Z_g.$ That is, choosing a different regular fiber at which to attach the singular $2$--handle yields an isotopic function on $X.$ Without loss of generality, assume $0 \notin Z_f$ and that $f^{-1}(0)$ is a fiber whose genus is maximal amongst regular fibers (i.e., $f^{-1}(0) \cong F_{g,b}$ is the trisection surface of genus $g$ with $b$ boundary components).

Let $\gamma: [0,1] \rightarrow D^2$ be a smoothly embedded path from $\gamma({\varepsilon})=\vec{0}$ to $\gamma(1-\varepsilon)=y_0$ such that: 
\begin{itemize}
	\item[(1)] $\gamma$ intersects $Z_f$ at points $p_1, \ldots, p_n \in D^2,$ none of which are cusps or crossings of $Z_f,$
	\item[(2)] if we denote $p_i = \gamma(t_i),$ then $t_i < t_{i+1}$ for every $i,$
	\item[(3)] the genus of the bounded fiber $f^{-1}(\gamma(t_i - \varepsilon))$ is one less than that of $f^{-1}(\gamma(t_i + \varepsilon).$ 
\end{itemize}
This gives us a path as in Figure~\ref{F:pushsing}. (The conditions above are simply to ensure that $\gamma$ is a path which does not intersect the same folds of $Z_f$ more than once.) Let $M_\gamma = f^{-1}(\gamma),$ then $\gamma^{-1} \circ f|_{M_\gamma}: M_\gamma \rightarrow [0,1]$ is a Morse function such that each $f^{-1}(p_i)$ and $f^{-1}(y_0)$ are index--$2$ critical points. It is a standard result in Morse theory that critical points of the same index can be reordered. That is, we can modify $f$ so that the index--$2$ critical point corresponding to the newly created Lefschetz singularity is attached to the fiber $f^{-1}(0).$ Finally, since it was arranged that $\gamma$ missed the cusps and crossings of $Z_f,$ we can extend the above construction to a neighborhood $N \subset D^2$ of $\gamma,$ which gives a perturbed map $f:X \rightarrow D^2$ with a single Lefschetz singularity with critical value $\vec{0}.$

\begin{figure}[ht!]
\labellist
	\pinlabel $Z_f$ at 75 122
	\pinlabel \resizebox{.8cm}{!}{$\longrightarrow$} at 195 65
	\pinlabel $\gamma$  at 373 92
\endlabellist
\centering
\resizebox{4.5in}{!}{
\includegraphics[scale=.6]{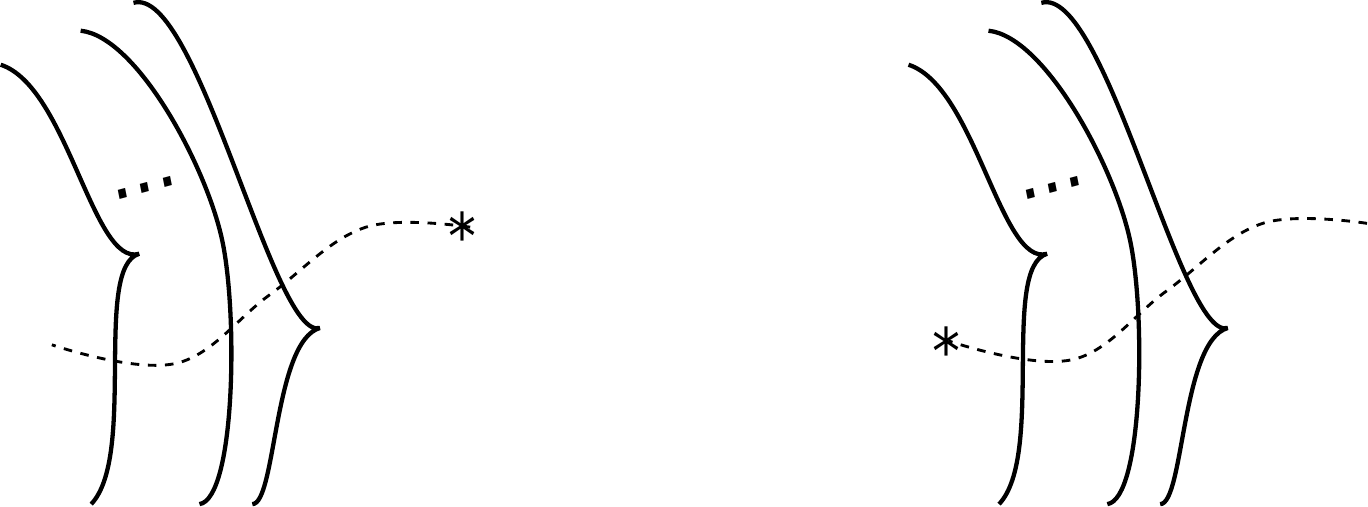}
}
	\caption{Preparing a Lefschetz Singularity for Wrinkling}
	\label{F:pushsing}
\end{figure}

Let us now perturb $f$ in a neighborhood of $\vec{0}$ via
	$$f_t(u,v) = u^2 +v^2 +tRe(u),$$
or in real coordinates
	$$f_t(s,x,y,z) = (s^2-x^2+y^2-z^2+ts,2sx _2yz).$$
For $t>0$ the critical values of $f_t$ are given by
	$$\Gamma_t:= \{(s,x,y,z) \in \mathbb{R}^4| x^2 +s^2 +\frac{st}{2} = 0, \; y=z=0\},$$
which defines a \emph{triple cuspoid} pictured in Figure~\ref{F:wrinkle}. In \cite{lekili}, Lekili shows that for $y_1, y_2 \in D^2$, where $y_1$ is in the interior of the triple cuspoid and $y_2$ is in the exterior, then the genus of $f_1^{-1}(y_1)$ is one greater than that of $f_1^{-1}(y_2).$ This perturbation is known as \emph{wrinkling a Lefschetz singularity.} The triple cuspoid can be thought of as a Cerf graphic, where each cusp  gives a canceling $1$--$2$ pair and crossing a fold into the bounded region corresponds to attaching a $1$--handle. Lekili further shows that crossing a fold into the exterior of the cuspoid corresponds to attaching a $2$--handle along a curve in the the central fiber. In Figure~\ref{F:wrinkle}, the colors of the attaching spheres correspond to the colors of $l_{ij} \subset D^2$ so as to indicate the isotopy class of curves determined by which fold of $\Gamma_t$ each line crosses. Notice that wrinkling is a local modification which is done on the interior of $X$. Thus, the action of wrinkling does not modify any boundary data.

\begin{figure}[th!]\centering
\hspace{.7in}
\resizebox{3.5in}{!}{%
\begin{tikzpicture}
	\draw [thick] (-8.7,.4) arc (180:360:.7 and .2);
	\draw [thick, dashed] (-7.31,.4) arc (0:180:.69 and .2);
	\draw [thick] (-8, 2.4) ellipse (.7cm and .2cm);
	
	\draw[thick] (-8.4, 1.4) arc (180:360:.4 and .1);
	\draw[thick, dashed] (-7.6, 1.4) arc (0:180:.4 and .1);
	
	\draw [thick] (-8.7, .4) .. controls (-8.3, 1.15) and (-8.3, 1.65) .. (-8.7, 2.4);
	\draw [thick] (-7.3, .4) .. controls (-7.7, 1.15) and (-7.7, 1.65) .. (-7.3, 2.4);

	\draw [thick] (-.7,0) arc (180:360:.7 and .2);
	\draw [thick, dashed] (.69,0) arc (0:180:.69 and .2);
	\draw [thick] (0, 3) ellipse (.7cm and .2cm);
	\draw [thick] (-.7, 0) .. controls (-.6, .3) and (-1.3, 1) .. (-1.27, 1.5);
	\draw [thick] (-1.27, 1.5) .. controls (-1.3, 2) and (-.6, 2.7) .. (-.7, 3);
	\draw [thick] (.7, 0) .. controls (.6, .3) and (1.3, 1) .. (1.27, 1.5);
	\draw [thick] (1.27, 1.5) .. controls (1.3, 2) and (.6, 2.7) .. (.7, 3);
	\draw [thick] (.1, 1) .. controls (-.25, 1.25) and (-.25, 1.75).. (.1, 2);
	\draw [thick] (0,1.1) .. controls (.25, 1.25) and (.25, 1.75).. (0,1.9);
	
	\draw [thick, red] (-1.27, 1.5) .. controls (-1, 1.4) and (-.44, 1.4) .. (-.172, 1.5);
	\draw [dashed, red] (-1.27, 1.5) .. controls (-1, 1.6) and (-.44, 1.6) .. (-.172, 1.5);
	
	\draw [thick, blue] (0, 1.5) ellipse (.6cm and 1cm);
	
	\draw [thick, green] (0, .3) .. controls (-.4,.28) and (-1.2, 1) .. (0, 1.9);
	\draw [thick,dashed, green] (1.16, 1.9) .. controls (.7, 3) and (-.8,2.3).. (0, 1.9);
	\draw [thick, green] (1.16, 1.9) .. controls (1, .3) and (.1, .28) .. (0, .3);
	
\end{tikzpicture}
}\newline
\resizebox{3.5in}{!}{%
\hspace{.4in}
\begin{tikzpicture}	
	\node at (-7, .3) {*};
	\node at (-7, .1) {$x_i$};
	\node at (-4,3) {$\xrightarrow{wrinkle}$};
	
	\draw [thick]plot [smooth, tension = 1] coordinates {(-1.1,.3) (.1, 1.7142) (.5,3.1)};
	\draw [thick](-8, 1) -- (-8.5, 3.3);
	\draw[thick] ([shift=(30:1cm)]-2,.866) arc (0:-60:2cm);
	\draw[thick] ([shift=(30:1cm)]-2,.866) arc (180:240:2cm);
	\draw[thick] ([shift=(30:1cm)]-1,-.866) arc (60:120:2cm); 
	
	\draw (-7, .3) circle (2cm);
	\draw (-1.1, .3) circle (2cm);
	\draw [thick, blue](-1.1, .3) -- (-1.1,-1.7);
	\draw [thick, green](-1.1, .3) -- (.632,1.3);
	\draw [thick, red](-1.1, .3) -- (-2.832,1.3);
	
\end{tikzpicture}
}
 	\caption{Wrinkling a Lefschetz Singularity}
	\label{F:wrinkle}
\end{figure}
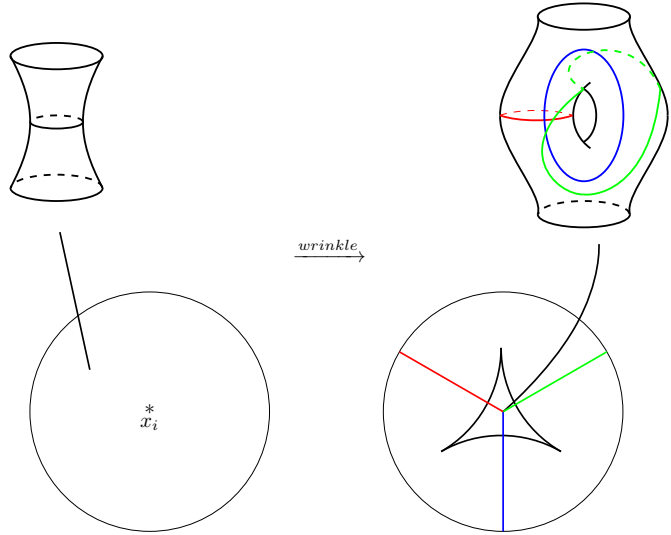
All that remains is to show that the resulting function does in fact result in a trisected $4$--manifold. Let $l_{ij}$ denote the image under the original map $f$ of $X_i \cap X_j$ fixed as a subset of $D^2.$ Moreover, let us arrange for the image of $X_1\cap X_2 \cap X_3$ to be $\vec{0}.$ For sufficiently small $t>0$, we may assume that $\Gamma_t$ is disjoint from $\Gamma_f$ and that each $l_{ij}$ do not intersect $\Gamma_t$ at a cusp or crossing. If we now choose an identification of $D^2$ with $I \times I$, we can proceed to construct a trisection $\tri'$ of $X$ as in Theorem~\ref{T:4dexistence}. 

\begin{definition}\label{D:relstab}
The above process is a \emph{stabilization} of the trisection $\tri$ relative to the open book $(P_i, \phi_i).$
\end{definition}

By construction, a relative stabilization of $\tri$ induces a stabilization of the open book decomposition $\tri|_{M} = (P_i, \phi_i).$ (Recall that this has the effect of plumbing a Hopf band onto the pages and the monodromy gets composed with a Dehn twist along the vanishing cycle.) If the feet of the $1$--handle are contained in a single binding component, then the plumbing increases the number of boundary components of the page by one and fixes the genus. If different binding components are involved, then the plumbing decreases the number of boundary components of the page by one and increases the genus by one. As mentioned before, wrinkling the Lefschetz singularity increases the genus of the central fiber by one.

Let us now summarize stabilizations of $\tri_X$ relative to $(P_i, \phi_i),$ resulting in a new trisection $\tri'_X.$
\begin{itemize}
	\item $\tri'$ admits a decomposition of $X= X'_1 \cup X'_2 \cup X'_3,$ where $X'_i \cong X_i \natural (S^1 \times B^3).$
	\item If $F_{g,b}$ is the trisection surface of $\tri,$ then the trisection surface of $\tri'$ is either $F_{g+1, b+1}$ or $F_{g+2, b-1}.$
	\item The induced open book $\tri'|_M = (P'_i, \phi'_i)$ is a Hopf stabilization of $(P_i, \phi_i).$
\end{itemize}

A complete uniqueness theorem for relative trisections would require a list of stabilizations which allow us to make any two trisection of a fixed $4$--manifold isotopic. It is unclear as to whether or not such a list exists (see Section~\ref{S:Conclusion}, Question~\ref{Q:DoubleTwist}). However, Gay and Kirby gave the following uniqueness statement:
\begin{theorem}[\cite{gay}]
	Any two relative trisections of $X$ which induce the same open books on $\partial X$ can be made isotopic after a finite number of interior stabilizations of both.
\end{theorem}
Now that we have relative stabilizations at our disposal, this statement can be strengthened.
\begin{theorem}\label{T:relativeuniqueness}
	If $\tri_1$ and $\tri_2$ are relative trisections of $X$ such that their induced open books on $\partial X$ can be made isotopic after Hopf stabilizations, then the two relative trisections can be made isotopic after a finite number of interior and relative stabilizations of each.
\end{theorem}
\begin{proof}
By assumption, we can perform relative stabilizations of $\tri_1$ and $\tri_2$ so that they induce equivalent open books on $\partial X.$ Since relative stabilizations, in some sense, ``factor through'' Lefschetz singularities, we have the liberty to choose the vanishing cycles of the associated singularities, thus allowing us to ensure that the appropriate Hopf stabilizations are induced. We now call upon the uniqueness statement of Gay and Kirby to finish the proof.
\end{proof}

Notice that relative stabilizations behave well with gluings due to the induced Hopf stabilization on the open book. More precisely, 
\begin{lemma}
Suppose $\tri_Z$ and $\tri_W$ relative trisections of $Z$ and $W$ with induced open books $(P, \phi)$ and $(Q, \psi)$ respectively. Let $f: \underset{r=1}{\overset{n}{\bigsqcup}}M_{\phi_{i_r}} \rightarrow \underset{r=1}{\overset{n}{\bigsqcup}}M_{\psi_{j_r}}$ be an orientation reversing diffeomorphism respecting the induced open books on each boundary component, where $\{i_1, \ldots, i_n\}, \{j_1, \ldots, j_n\} \subset \{1, \ldots, m\}.$ If $\tri'_Z$ and $\tri'_W$ are relative stabilizations of $\tri_Z$ and $\tri_W$ relative to $(P_{i_1}, \phi_{i_1})$ and $(Q_{j_1}, \psi_{j_1})$ respectively such that the new induced open books remain compatible, then $f$ can be extended so as to induce the trisection $\tri'_Z \underset{f}{\cup} \tri'_W$ on $Z \underset{f}{\cup} W.$
\end{lemma}
The proof of the lemma is straight forward and left to the reader. 
\end{section}

\begin{section}{Final Remarks \label{S:Conclusion}}
As mentioned in Section~\ref{S:Background}, we can easily obtain a relative trisection from a Lefschetz fibration over a disk with bounded fibers. To do this, wrinkle a single Lefschetz singularity, then move the remaining singularities through the indefinite folds coming from the wrinkling as in Figure~\ref{F:pushsing}. Repeating this process until no Lefschetz singularities remain results in a fibration whose singular values are nested triple cuspoids in $D^2.$  A diagramatic version of this process is given in \cite{cgp}.

Theorems~\ref{T:existence} and \ref{T:Gluing} allow us to define a category of trisections \textbf{Tri} whose object are closed $3$--manifolds, either connected or disconnected, equipped with an open book decomposition and whose morphisms are relatively trisected $4$--manifolds up to isotopy and interior stabilizations. Associativity is granted to us by the gluing theorem and each object has an identity since any two relative trisections of $X$ inducing the same open book decomposition(s) on $\partial X$ are stably equivalent.

\begin{question}\label{Q:DoubleTwist} One might hope for a full uniqueness result of relative trisections which strengthens Theorem~\ref{T:relativeuniqueness} by removing the condition on the induced open book decompositions. This would require a collection of operations on relative trisections such that the induced moves on the bounding $3$--manifold allows us to relate any two open book decomposition. In the case that $\partial X$ is a rational homology sphere, relative stabilizations are sufficient since any two open books of a rational homology sphere equivalent under Hopf stabilization. However, when dealing with an arbitrary $3$--manifold, Harer provided us with the \emph{double twist} in \cite{harer} (which was shown to be redundant when dealing with $S^3).$ This double twist would have to be realized as being induced by a move on relative trisections. It is clear that this can be done by taking the connected sum of two trisected $\mathbb{C}P^2$ or $\overline{\mathbb{C}P^2}.$ but we can ask if it is possible to induce a double twist on $\partial X$ while fixing the diffeomorphism type of the $X$.
\end{question}
\end{section}
\newpage

\bibliography{references}

\begin{thebibliography}{1}

\bibitem{alexander}
James~W. Alexander.
\newblock Note on {R}iemann spaces.
\newblock {\em Bull. Amer. Math. Soc.}, 26(8):370--372, 1920.

\bibitem{cgp}
Nickolas~A. Castro, David~T. Gay, and Juanita Pinz\'{o}n-Caicedo.
\newblock Diagrams for relative trisections.
\newblock {\em arXiV:1610.06373}, 2016.

\bibitem{etnyre}
John~B. Etnyre.
\newblock Lectures on open book decompositions and contact structures.
\newblock In {\em Floer homology, gauge theory, and low-dimensional topology},
  volume~5, pages 103--141. Clay Math Proc., 2006.

\bibitem{gay}
David~T. Gay and Robion~C. Kirby.
\newblock Trisecting $4$--manifolds.
\newblock {\em Geom. Top.}, 20:3097--3132, 2016.

\bibitem{gompf}
Robert~E. Gompf and Andr{\'a}s~I. Stipsicz.
\newblock {\em {$4$}-manifolds and {K}irby calculus}, volume~20 of {\em
  Graduate Studies in Mathematics}.
\newblock American Mathematical Society, Providence, RI, 1999.

\bibitem{harer}
John Harer.
\newblock How to construct all fibered knots and links.
\newblock {\em Topology}, 21(3):263--280, 1982.

\bibitem{lekili}
Yanki Lekili.
\newblock Wrinkled fibrations on near-symplectic manifolds.
\newblock {\em Geom. Topol.}, 13(1):277--318, 2009.
\newblock Appendix B by R. {\.I}nan{\c{c}} Baykur.

\bibitem{osb}
Burak Ozbagci and Andr{\'a}s~I. Stipsicz.
\newblock {\em Surgery on contact 3-manifolds and {S}tein surfaces}, volume~13
  of {\em Bolyai Society Mathematical Studies}.
\newblock Springer-Verlag, Berlin; J\'anos Bolyai Mathematical Society,
  Budapest, 2004.

\end{thebibliography}
\bibliographystyle{plain}

\end{document}